\documentclass[11pt]{article}
\usepackage{epsfig,amsmath,amssymb,latexsym,color}
\usepackage{graphicx,subfig,commath,fixmath}
\usepackage[shortlabels]{enumitem}
\usepackage{ifthen}
\usepackage{setspace}

\newtheorem{theorem}{Theorem}
\newtheorem{lemma}{Lemma}

\newtheorem{remark}{Remark}

\newenvironment{proof}{\begin{trivlist}\item[]{\emph{Proof.}}}
               {\hfill$\Box$\end{trivlist}}

\def\bfx{{\bf x}}
\def\bfb{{\bf b}}
\def\bfE{{\bf E}}
\def\bfa{{\bf a}}
\def\bfr{{\bf r}}

\usepackage{hyperref}

\begin{document}
\title{A Quasi-Orthogonal Matching Pursuit Algorithm\\ for Compressive Sensing}
\author{Ming-Jun Lai\footnote{mjlai@uga.edu. This author is associated with Department of Mathematics, University of
Georgia, Athens, GA 30602, U.S.A.} \and Zhaiming Shen\footnote{Zhaiming.Shen@uga.edu. 
This author is associated with Department of Mathematics, University of
Georgia, Athens, GA 30602, U.S.A.}}
\maketitle

\begin{abstract}
\noindent
In this paper, we propose a new orthogonal matching pursuit algorithm called quasi-OMP 
algorithm which greatly enhances the performance of classical orthogonal matching pursuit 
(OMP) algorithm, at some cost of computational complexity.
We are able to show that under some sufficient conditions of mutual coherence 
of the sensing matrix, 
the QOMP Algorithm succeeds in recovering the $s$-sparse signal vector $\mathbf{x}$ within $s$ 
iterations where a total number of $2s$ columns are selected under the both noiseless and noisy settings. In addition, we
show that for Gaussian sensing matrix, the norm of the residual of each iteration will go to zero linearly depends 
on the size of the matrix with high probability. The numerical experiments are 
demonstrated to show the effectiveness of QOMP algorithm in recovering sparse 
solutions which outperforms the classic OMP and GOMP algorithm.
\end{abstract}

\section{Introduction}
The problem we discuss in this paper is the following:
Suppose we are given the sensing matrix $\mathbf{\mathbf{\Phi}}
\in\mathbb{R}^{m\times n}$,  the observed 
measurement vector $\mathbf{b}$, and the sparsity $s$ of the input signal vector $\mathbf{x}$, 
let $S$ denote the support of $\mathbf{x}$. Given $\mathbf{\mathbf{\Phi}}$, $\mathbf{b}$ and $s$, in the 
noiseless case and noisy case, we would like to recover the original input signal vector $\mathbf{x}$ through 
the equation $\mathbf{\mathbf{\Phi}}\mathbf{x}=\mathbf{b}$. This is a typical problem in the study of
compressive sensing. That is, we solve 
\begin{equation}
\label{CS}
\min_{\bfx  \in \mathbb{R}^n} \{\|\mathbf{\mathbf{\Phi}}\mathbf{x}- \bfb\|_2: \quad \|\bfx\|_0\le s\},
\end{equation}
where $\|\bfx\|_0$ stands for the number of nonzero entries of $\bfx$. 
A commonly useful concept which plays a key role 
in the study of the existence and uniqueness of a sparse solution from a sensing matrix 
 is called \emph{Restricted Isometry Constant} (RIC), which is defined as, for a sensing matrix 
$\mathbf{\mathbf{\Phi}}\in\mathbb{R}^{m\times n}$ and an integer $1\leq s\leq n$, the smallest constant 
$\delta_s\in (0,1)$ such that 
\begin{equation} \label{eqn:RIP}
(1-\delta_s)\|\mathbf{x}\|^2_2\leq \|\mathbf{{\Phi} x}\|^2_2\leq (1+\delta_s)\|\mathbf{x}\|^2_2
\end{equation}
for all $s$-sparse signals $\mathbf{x}$. If a matrix $\mathbf{\Phi}$ satisfies (\ref{eqn:RIP}), then we say $\mathbf{\Phi}$ is of \emph{Restricted Isometry Property} (RIP) of order $s$ with restricted isometry constant $\delta_s$. 
\\
Another useful concept in compressive sensing study
is the \emph{mutual coherence}, which characterizes the spread 
of the columns of $\mathbf{\Phi}$. The mutual coherence of a sensing matrix $\mathbf{\Phi}$ is defined as
\begin{equation}
\mu(\mathbf{\Phi}):=\max_{1\leq i,j\leq n, i\neq j}
\frac{|\mathbold{\mathbold{\phi}}_i^\top\mathbold{\mathbold{\phi}}_j|}{\|\mathbold{\mathbold{\phi}}_i\|_2\cdot \|\mathbold{\mathbold{\phi}}_j\|_2}.
\end{equation}  
Let $\mathbold{\phi}_i$ denote the column of $\mathbf{\Phi}$. If each $\mathbold{\phi}_i$ is normalized, then we have 
$\mu(\mathbf{{\Phi}})=\max_{1\leq i,j\leq m, i\neq j}|\mathbold{\mathbold{\phi}}_i^\top\mathbold{\mathbold{\phi}}_j|$.
\\
In compressive sensing study, the OMP algorithm is one of the most important approaches.  
The main idea of the OMP algorithm, which is also called orthogonal greedy 
approach, can be 
explained as follows. We can be greedy enough to set $s=1$ in (\ref{CS}). Then the minimization
in (\ref{CS}) becomes
\begin{equation}
\label{laisidea}
\min_{i, c} \|c \mathbold{\mathbold{\phi}}_i -\bfb\| =\min_i \min_c \|c \mathbold{\mathbold{\phi}}_i -\bfb\|,
\end{equation}
where $\mathbf{\Phi}=[\mathbold{\mathbold{\phi}}_1, \mathbold{\mathbold{\phi}}_2, \cdots, \mathbold{\mathbold{\phi}}_n]$. 
The inner minimization problem is easy to solve 
and under the assumption that $\|\mathbold{\mathbold{\phi}}_i\|=1$, we know $c= \langle \bfb, 
\mathbold{\mathbold{\phi}}_i\rangle$. That is, the residual value
\begin{equation}
\label{laisidea2}
\min_{i, c} \|c \mathbold{\mathbold{\phi}}_i -\bfb\|^2 = \min_i \|\bfb\|^2 - |\langle \bfb, \mathbold{\mathbold{\phi}}_i\rangle|^2
\end{equation} 
will be minimized if the index $i$ is chosen such that $ |\langle \bf{b}, \mathbold{\mathbold{\phi}}_i\rangle| 
= \max_{j} |\langle \bfb, \mathbold{\mathbold{\phi}}_j\rangle|$. Certainly, we then repeat the previous procedure
by letting $\bfr_1=\bfb- \langle \bfb, \mathbold{\mathbold{\phi}}_i\rangle \mathbold{\mathbold{\phi}}_i$ be the new 
residual vector and computing the next index $i_1$ 
so that $|\langle \bfb_1, \mathbold{\mathbold{\phi}}_i\rangle|$ is largest.
 Continue this procedure until a certain number of iterations or a 
certain stopping criterion for the $k$-th residual $\bfr_k$ is achieved. 
These steps form the so-called OMP algorithm which is summarized as the following. \\
\rule{\textwidth}{0.4pt}
\textbf{Algorithm 1:} Orthogonal Matching Pursuit (OMP) \\
\rule{\textwidth}{0.4pt} 
\textbf{Input:} $\mathbf{\mathbold{\phi}}_{m\times n}$, $\mathbf{b}_{n\times 1}$, sparsity $s$, maximum iterations $k_{\max}$ ($k_{\max}<m$), and tolerance $\epsilon$. \\
\textbf{Initialization:} $\mathcal{S}_0=\emptyset$, $\mathbf{r}_0=\mathbf{b}$, $k=0$.\\
\vspace{0.1cm}
\ \ \ \textbf{while} $k<k_{\max}$ and $|r_k|>\epsilon$ \\
\vspace{0cm}
\ \ \ \ \ \ \ \ $k=k+1$; \\
\vspace{0cm}
\ \ \ \ \ \ \ \ $i_k=arg\max_{1\leq i\leq n}\{|\mathbf{\mathbold{\phi}_i}^\top\mathbf{r_{k-1}}|\}$; \\
\vspace{0cm}
\ \ \ \ \ \ \ \ $\mathcal{S}_k=\mathcal{S}_{k-1}\cup\{i_k\}$; \\
\vspace{0cm}
\ \ \ \ \ \ \ \ $\mathbf{r_k}=\mathbf{b}-\mathbf{\Phi}_{\mathcal{S}_k}\mathbf{\Phi}_{\mathcal{S}_k}^{\dagger}\mathbf{b}$; \\
\vspace{0cm}
\ \ \ \textbf{end} \\
\vspace{0cm}
\textbf{Output:} $\mathcal{S}=\mathcal{S}_{k}$, $\mathbf{x}_{S}=\mathbf{\Phi}_{\mathcal{S}}^{\dagger}\mathbf{b}$, and $\mathbf{x}_{S^c}=\mathbf{0}$. \\
\rule{\textwidth}{0.4pt} \\  
Since the orthogonal matching pursuit (OMP)  algorithm for compressive sensing study 
was introduced \cite{TroppFirst}, plenty 
of the different modified OMP algorithms have been developed. 
For example, the regularized OMP (ROMP) \cite{NeedellROMP}, generalized OMP (GOMP) \cite{WangGOMP}, 
stagewise OMP (StOMP) \cite{DonohoStOMP}, subspace pursuit (SP) \cite{DaiSP}, and compressive 
sampling matching pursuit (CoSaMP) \cite{NeedellCosamp}. 
The following generalized orthogonal matching pursuit (GOMP) is a generalization of the standard OMP.  \\
\rule{\textwidth}{0.4pt}
\textbf{Algorithm 2:} Generalized Orthogonal Matching Pursuit (GOMP) \\
\rule{\textwidth}{0.4pt}
\textbf{Input:} $\mathbf{\Phi}_{m\times n}$, $\mathbf{b}_{n\times 1}$, sparsity $s$, number of indices $N$ for each iteration, maximum iterations $k_{\max}$ ($k_{\max}\leq m/N$), and tolerance $\epsilon$. \\
\textbf{Initialization:} $\mathcal{S}_0=\emptyset$, $\mathbf{r}_0=\mathbf{b}$, $k=0$.\\
\vspace{0.1cm}
\ \ \ \textbf{while} $k< k_{\max}$ and $|r_k|>\epsilon$\\
\vspace{0cm}
\ \ \ \ \ \ \ \ $k=k+1$; \\
\vspace{0cm}
\ \ \ \ \ \ \ \ $\{i_1,i_2,\cdots,i_N\}=$ the largest $N$ indices which maximize $|\mathbold{\phi}_i^\top\mathbf{r}_{k-1}|$; \\
\vspace{0cm}
\ \ \ \ \ \ \ \ $\mathcal{S}_k=\mathcal{S}_{k-1}\cup\{i_1,i_2,\cdots,i_N\}$; \\
\vspace{0cm}
\ \ \ \ \ \ \ \ $\mathbf{r_k}=\mathbf{b}-\mathbf{\Phi}_{\mathcal{S}_k}\mathbf{\Phi}_{\mathcal{S}_k}^{\dagger}\mathbf{b}$; \\
\vspace{0cm}
\ \ \ \textbf{end} \\
\vspace{0cm}
\textbf{Output:} $\mathcal{S}=\mathcal{S}_{k}$, $\mathbf{x}_{S}=\mathbf{\Phi}_{\mathcal{S}}^{\dagger}\mathbf{b}$, and $\mathbf{x}_{S^c}=\mathbf{0}$. \\
\rule{\textwidth}{0.4pt} \\
\\
The family of OMP algorithms
have largely drawn people's attention because of its effective performance and 
its high efficiency. A lot of theoretical aspects of the OMP algorithm has been developed as 
well. In \cite{TroppOMP}, Tropp and Gilbert showed that, for a $s$-sparse signal $\mathbf{x}$ 
and an Gaussian sensing matrix $\mathbf{\Phi}\in\mathbb{R}^{m\times n}$, the OMP recovers 
$\mathbf{x}$ from $\mathbf{y} = \mathbf{\Phi}\mathbf{x}$ with overwhelming probability if the 
number of measurements satisfy $m\sim s\cdot log n$. In \cite{WangOMP}, Wang and Shim showed 
that the exact recovery of an $s$-sparse signal can be guaranteed by using OMP algorithm in $s$ 
iterations if the RIP constant satisfies $\delta_{s+1}<\frac{1}{\sqrt{s}+1}$.
See also \cite{MoShen12} for the same result.  More recently,  
this condition has been improved to $\delta_{s+1}<\frac{1}{\sqrt{s+1}}$ by Mo in 
\cite{MoSharpbound} and \cite{WenMoSharpcondition}, and extended to block signal recovery setting in 
\cite{LaiBlocksparsesignal}. Meanwhile, the bound $\delta_{s+1}<\frac{1}{\sqrt{s+1}}$ is also strict, as 
it was shown in \cite{WenImprovedbound} and \cite{WenMoSharpcondition},
the OMP algorithm may fail to recover $s$-sparse signal 
$\mathbf{x}$ in $s$ iterations if $\delta_{s+1}\geq\frac{1}{\sqrt{s+1}}$. Researchers are also interested in finding the 
conditions of eventually recovering the $s$-sparse signal with more than $s$ iterations. 
In \cite{ZhangSparserecovery}, Zhang showed that OMP recovered any $s$-sparse 
signal with $30s$  iteration under the condition $\delta_{31s}<\frac{1}{3}$. See also \cite{Cohen17} for another convergence
analysis of the OMP iterations.  Zhang's result has recently been 
improved by Wang and Shim in \cite{WangHowmanyiterations}, which showed that OMP can accurately 
recover all $s$-sparse signals within $\left\lceil 2.8s\right\rceil$ iterations if the matrix 
$\mathbf{\mathbold{\phi}}$ satisfies a certain restricted isometry property (RIP) condition. There are many other results on the 
OMP algorithms in the literature and we will not exhaust them in this paper. 
\\
\\
We propose in this paper a new approach, which we will call it 
\emph{quasi-orthogonal matching pursuit}
(QOMP) algorithm. The main idea of QOMP is described as follows. 
\\
Instead of choosing $s=1$ each time in 
(\ref{laisidea}), we are greedy enough to choose $2$ terms as sparse solutions 
 since most applications have a sparsity more than 2. We have to solve the best  approximations
\begin{equation}
\label{Shensidea}
 \min_{b_1, b_2} \|b_1 \mathbold{\mathbold{\phi}}_i+b_2\mathbold{\phi}_j -\bfb\|.
\end{equation}
for all $i\not =j, i, j=1, \cdots, n$ to find the residuals. 
We choose the best index pair, $(i_1, j_1)$ such that the residual is the smallest:  
\begin{equation}
\label{QOMPidea}
\min_{i, j, b_1, b_2\atop i\not=j} 
\|b_1 \mathbold{\phi}_i+b_2\mathbold{\phi}_j -\bfb\| = \min_{b_1, b_2} \|\bfb - b_1 \mathbold{\phi}_{i_1} - b_2\mathbold{\phi}_{j_1}\|. 
\end{equation}
Once we find $(i_1,j_1)$ to solve (\ref{QOMPidea}), 
we let $\bfr_1= \bfb- b_{i_1}\mathbold{\phi}_{i_1}-
b_{j_1}\mathbold{\phi}_{j_1}$ and repeat the procedure.  
This leads to our QOMP algorithm. It is worthwhile to note that QOMP becomes GOMP with $N=2$ if all the columns 
$\mathbold{\phi}_i$ are orthogonal to each other, as the minimization problems (\ref{Shensidea}) and (\ref{QOMPidea}) 
decouples to find the two indices which maximize $|\mathbold{\phi}_i^\top\mathbf{b}|$. 
\\
\\
Clearly, the computational burden is significantly increased. 
However, due to the parallel computation or graphics processing unit (GPU) setting  as  the amount of
computation for each minimization in (\ref{Shensidea}) is small, 
one is able to carry out the computation when $n$ is reasonably large, say $n=1000--10,000$. See \S 3.1 for our explanation. 
This  also explains a significant difference from the weak OMP, OMMP, and BOMP algorithms 
as multiple indices are chosen during each iteration, 
see, e.g. \cite{TroppFirst}, \cite{Xu15}, and \cite{LaiBlocksparsesignal}. 
The QOMP algorithm is summarized as following. \\
\rule{\textwidth}{0.4pt}
\textbf{Algorithm 3:} Quasi-Orthogonal Matching Pursuit (QOMP) \\
\rule{\textwidth}{0.4pt}
\textbf{Input:} $\mathbf{\Phi}_{m\times n}$, $\mathbf{b}_{n\times 1}$, sparsity $s$ ($s\geq 2$), maximum 
iterations $k_{\max}$ ($k_{\max}\leq m/2$), and tolerance $\epsilon$. \\
\textbf{Initialization:} $\mathcal{S}_0=\emptyset$, $\mathbf{r}_0=\mathbf{b}$, $k=0$. 
$\mathbf{\Psi}_{m\times n}=\mathbf{\Phi}_{m\times n}$. \\
\vspace{0.1cm}
\ \ \ \textbf{while} $k<k_{\max}$ and $|\mathbf{r_k}|>\epsilon$ \\ 
\vspace{0cm}
\ \ \ \ \ \ \ \ $k=k+1$;  \\
\vspace{0cm}
\ \ \ \ \ \ \ \ 
$\mathbf{Res}_{(i,j)}(\mathbf{r}_{k-1})=\min_{u,v\in\mathbb{R}}\{\|\mathbold{\psi}_i u
+\mathbold{\psi}_j v-\mathbf{r_{k-1}}\|_2\}$;  \\
\vspace{0cm}
\ \ \ \ \ \ \ \ $(i_k,j_k)=arg\min_{1\leq i\leq n, 1\leq j
\leq n}\{\mathbf{Res}_{(i,j)}(\mathbf{r}_{k-1})\}$; \\
\vspace{0cm}
\ \ \ \ \ \ \ \ $\mathcal{S}_k=\mathcal{S}_{k-1}\cup\{i_k,j_k\}$; \\
\vspace{0cm}
\ \ \ \ \ \ \ \ 
$\mathbf{r_k}=\mathbf{b}-\mathbf{{\Phi}}_{\mathcal{S}_k}\mathbf{{\Phi}}_{\mathcal{S}_k}^{\dagger}
\mathbf{b}$; \\
\vspace{0cm}
\ \ \ \ \ \ \ \ $\mathbf{\Psi}_{\{i_k,j_k\}}=\mathbf{0};$ \\
\vspace{0cm}
\ \ \ \textbf{end} \\
\vspace{0cm}
\textbf{Output:} $\mathcal{S}=\mathcal{S}_{k}$, 
$\mathbf{x}_{S}=\mathbf{{\Phi}}_{\mathcal{S}}^{\dagger}\mathbf{b}$, and 
$\mathbf{x}_{S^c}=\mathbf{0}$. \\
\rule{\textwidth}{0.4pt} 
\\
\\
The notation $\mathbf{\Phi}_{\mathcal{S}}^{\dagger}=(\mathbf{\Phi}_{\mathcal{S}}^\top\mathbf{{\Phi}}_{\mathcal{S}})^{-1}\mathbf{{\Phi}}_{\mathcal{S}}^\top$ is the pseudo-inverse of $\mathbf{\Phi}_{\mathcal{S}}$, and $\mathbf{Res}_{(i,j)}(\mathbf{r}_{k-1})$ is the residual of $\mathbf{r}_{k-1}$ after projected onto the hyperplane spanned by the columns $\mathbf{\mathbold{\phi}}_i$ and $\mathbf{\mathbold{\phi}}_j$.  Note that the maximum iterations can not exceed $\frac{m}{2}$, otherwise the pseudo-inverse would not make sense.   Also note that we have a column update step $\mathbf{\Psi}_{\{i_k,j_k\}}=\mathbf{0}$, it is because we do not want the process to pick the same indices as the previous iterations. As seen in the algorithm above, 
instead of indexing each column of the sensing matrix, we 
index a pair of columns of $\mathbf{\Phi}$ from each iteration. 
In the $k$-th iteration step, the algorithm filters 
in the pair of columns of the largest correlation with the current vector 
measurement from our sensing matrix, and then add these two indices of columns 
in that pair as two new elements to the current support set $\mathcal{S}_{k-1}$.  
\\
\\
Suppose the sparsity of $\mathbf{x}$ is $s$, since each time two columns of the sensing matrix 
are chosen by the algorithm,  
we need at least $\left\lceil \frac{s}{2} \right\rceil$ total 
number of iterations.
There are chances that the algorithm fails to exactly recover $\mathbf{x}$ within $\left\lceil \frac{s}{2} \right\rceil$ number of iterations, for example, if the sparsity $s$ of the signal vector $\mathbf{x}$ becomes large.
We remedy this by adding two steps to ensure that the QOMP algorithm can perform well. Firstly, we add $\left\lfloor \frac{s}{2} 
\right\rfloor$ more iterations in addition to $\left\lceil \frac{s}{2} \right\rceil$ iterations (hence a total of $s$ iterations 
with a total of $2s$ columns being selected) 
in the algorithm to get a superset 
$\mathcal{S}$ which hopefully will contain the support of the true sparse solution 
$\mathbf{x}$. Secondly, we use $\mathcal{S}$ as the 
index set to get the estimated signal $\mathbf{\hat{x}}$ by using a sparse least square method (i.e. greedy QR decomposition for
$\Phi$ to obtain the least square solution  instead of SVD/pseudo-inverse), which should 
have $r\le 2s$ nonzero entries if the rank of $\Phi_{\cal S}$ is $r<2s$ or $s$ nonzero entries of $\Phi_{\cal S}$ is of full 
rank. Thus we can approximate $\mathbf{x}$ by $\mathbf{\hat{x}}$ as long as  the infinity norm of 
$\mathbf{\Phi}\mathbf{\hat{x}} -\bfb$ is negligible. 
\\
\\
The organization of this paper is as follows. We shall first establish  the QOMP algorithm 
in the next section by showing that each iteration finds at least one correct
index, and hence the exact recover of the $s$-sparse signal $\mathbf{x}$ can be guaranteed within $s$ iterations. 
Then we show in the noisy setting, the QOMP will find the 
correct indices if the noisy level is small relative to the smallest nonzero
entry of the exact signal $\bfx$.  
Next we show the norm of the residual vector ${\bf r}_k$ decreases to zero in 
a linear fashion (depends on the size of the matrix) with high probability.  
Furthermore we shall demonstrate that the new algorithm has a 
better performance than the standard OMP and GOMP (with $N=2$)  numerically in Section \S 3. 
In addition, we shall explain that the computational complexity of the QOMP 
is reasonable when the size $n$ of the signal $\bfx$ is not too large based on 
parallel computation or GPU.    Finally, we make some comments and point out future research problems 
in the Section \S 4.

\section{Theoretical Analysis of Convergence} 
\subsection{Signal Recovery in the Noiseless Setting}
From this section on, we assume all the norm 
$\|\cdot\|$ is $2$-norm if without specification, 
and all the dimension of a matrix of size $m\times n$ satisfies $m<n$. 
There are plenty of conditions have been developed to imposed on the 
restricted isometry 
constant $\delta$ of the sensing matrix in order to have  a better performance 
of recovering the 
sparse signal vector, as we already see in the introduction. In the present 
paper, we will shed 
more light on the mutual coherence $\mu$ of the sensing matrix $\mathbf{\Phi}$. \\
It is known that if $m<n$, then $\mu(\mathbf{{\Phi}})\geq 
\sqrt{\frac{n-m}{m(n-1)}}\geq \frac{1}{2m}$ if $m$ is large enough, see 
\cite{lowerboundmutualcoherence}. However, in our case, we are interested in 
finding an upper 
bound for the mutual coherence $\mu(\mathbf{\Phi})$ if the given 
sensing matrix   
$\mathbf{{\Phi}}$ is randomly generated with each entry of $\mathbf{\Phi}$ 
is i.i.d. of mean zero with  
some finite variance. For example, letting $\phi_{ij}$ denote the entries of $\mathbf{{\Phi}}$,  
then $\phi_{ij}\sim \mathcal{N}(0,1)$ or $\phi_{ij}\sim Unif(0,1)$. If each entry of $\mathbf{\Phi}$ is i.i.d. from the standard Gaussian distribution, then we say that the sensing matrix $\mathbf{{\Phi}}$ is Gaussian.\\
Let us first state a version of strong laws of large numbers 
which we will use later to prove our results.
\begin{lemma} (Strong laws of large numbers, Kolmogorov, Marcinkiewicz 
and Zygmund) 
\label{lem1}
Let $\xi,\xi_1,\xi_2,\cdots$ be i.i.d. random variables, and fix any $p\in (0,2)$. Then 
$n^{-1/p}\sum_{k\leq n}\xi_k$ converges a.s. 
iff ${\bf E}(|\xi|^p)<\infty$ and either $p\leq 1$ or 
${\bf E}(\xi)=0$. 
In that case, the limit equals ${\bf E}(\xi)$ for $p=1$ and is otherwise $0$.
\end{lemma}
\begin{proof}
We leave the proof to Appendix A or refer to Theorem $3.23$ in 
\cite{KallenbergProbability} for a proof.
\end{proof}
Governed by the strong laws of large numbers, the mutual coherence 
$\mu(\mathbf{{\Phi}})$ will 
decrease in the order comparable to $1/o(\sqrt{m})$ as the size of 
$\mathbf{{\Phi}}$ becomes large. 
 
\begin{lemma} \label{lem2}
Let $\mathbf{\Phi}\in\mathbb{R}^{m\times n}$ be a Gaussian sensing matrix. Then for large $m$, we have $\mu(\mathbf{{\Phi}})\leq \frac{1}{f(m)}$ for some function $f$, 
where $f(m)=o(\sqrt{m})$ and $f(m)\to\infty$ as $m\to\infty$, with high probability (e.g. $f(m)=\sqrt{m}/\log(m)$).
\end{lemma}

\begin{proof}
Firstly, we observe that when $m$ gets larger, the norm $\|\mathbf{\mathbold{\phi}}_i\|$ is comparable to $\sqrt{m}$. Indeed, as for each 
$\phi_{ij}$ with finite variance, we 
have $\|\mathbf{\mathbold{\phi}}_i\|^2=\sum_{j=1}^m \phi^2_{ji}$. \\ Since each $\phi_{ji}\sim\mathcal{N}(0,1)$, we 
have $\bfE(\phi^2_{ji})=1$ for all $j=1,2,\cdots,m$, hence by the strong law of large numbers, 
$\frac{1}{m}\|\mathbf{\mathbold{\phi}}_i\|^2=\frac{1}{m}\sum_{j=1}^m \phi^2_{ji}\to 1$ 
a.s., hence $\|\mathbf{\mathbold{\phi}}_i\|\to\sqrt{m}$ a.s. as $m\to\infty$. \\
So now we have $\mu(\mathbf{{\Phi}})=\max_{1\leq i,j\leq n, i\neq 
j}\frac{|\mathbf{\mathbold{\phi}}_i^\top\mathbf{\mathbold{\phi}}_j|}{\|\mathbf{\mathbold{\phi}}_i\|\cdot 
\|\mathbf{\mathbold{\phi}}_j\|}\approx\max_{1\leq 
i,j\leq n, i\neq j}\frac{|\mathbf{\mathbold{\phi}}_i^\top\mathbf{\mathbold{\phi}}_j|}{m}$ as 
$m\to \infty$. Note that 
$|\mathbf{\mathbold{\phi}}_i^\top\mathbf{\mathbold{\phi}}_j|=|\sum_{k=1}^m \phi_{ki}\phi_{kj}|$, by letting 
$X_k=\phi_{ki}\phi_{kj}$, we 
have 
$\frac{1}{m}\cdot|\mathbf{\mathbold{\phi}}_i^\top\mathbf{\mathbold{\phi}}_j|=\frac{1}{m}
\cdot|\sum_{k=1}^m X_k|$. \\
By the independence of $\phi_{ki}$ and $\phi_{kj}$, the expectation of each 
$X_k$ satisfies
\begin{align*}
\bfE(X_k)=\bfE(\phi_{ki}\phi_{kj})=\bfE(\phi_{ki})\bfE(\phi_{kj})=0,
\end{align*}
and the variance of each $X_k$ satisfies 
\begin{align*}
Var(X_k)=& Var(\phi_{ki}\phi_{kj})\cr 
=&(\bfE(\phi_{ki}))^2 Var(\phi_{kj})+(\bfE(\phi_{kj}))^2 Var(\phi_{ki})+Var(\phi_{ki})Var(\phi_{kj})=1,
\end{align*}
which implies that $\bfE(|X_k|^2)=\bfE(X_k^2)=Var(X_k)+(\bfE(X_k))^2=1$. Since the measure of a 
probability space is always one, which is bounded, we then have $L^p(X_k)\subset L^2(X_k)$ as 
$p<2$ for all $k$. Therefore $\bfE(|X_k|^p)<\infty$ for $p<2$ and for all $k$. \\
Now apply Lemma \ref{lem1} to the sequence of random variables $X_k$, $k=1,2,\cdots,n$, we have 
\begin{equation}
m^{-1/p}(\sum_{k=1}^m X_k)=m^{-1/p}(\sum_{k=1}^m \phi_{ki}\phi_{kj})=
m^{-1/p}(\mathbf{\mathbold{\phi}}_i^\top\mathbf{\mathbold{\phi}}_j)\to 0
\end{equation}
almost surely  for all $p\in (0,2)$. Since the limit is zero, so is 
\begin{equation}
m^{-1/p}|\sum_{k=1}^m X_k|=m^{-1/p}|\sum_{k=1}^m \phi_{ki}\phi_{kj}|
=m^{-1/p}|\mathbf{\mathbold{\phi}}_i^\top\mathbf{\mathbold{\phi}}_j|\to 0.
\end{equation}
Note that $f(m)=o(\sqrt{m})$, there is some $p<2$ such that
$f(m)\cdot m^{\frac{1}{p}-1}$ 
converges to zero almost surely. For example, we can take $\frac{1}{p}=3/4-1/2\cdot 
\log_m f(m)$, as $f(m)=o(\sqrt{m})$, we will have $log_m f(m) <
\log_m {\sqrt{m}}=1/2$ for 
large $m$, therefore $\frac{1}{p}=3/4-1/2\cdot \log_m f(m) 
>3/4-1/2\cdot \log_m \sqrt{m}
=1/2$, or $p<2$. By plugging $p$ into $f(m)\cdot m^{\frac{1}{p}-1}$, we have $f(m)\cdot 
m^{\frac{1}{p}-1}=(\frac{f(m)}{\sqrt{m}})^{1/2}\to 0$ as $m\to\infty$.  Hence we have  
\begin{equation}
m^{-1}\cdot|\mathbf{\mathbold{\phi}}_i^\top\mathbf{\mathbold{\phi}}_j|\cdot f(m)
=m^{-1/p}\cdot|\mathbf{\mathbold{\phi}}_i^\top\mathbf{\mathbold{\phi}}_j|
\cdot f(m)\cdot m^{\frac{1}{p}-1}\to 0
\end{equation}
almost surely as $m\to\infty$. Therefore, by taking the supremum over all $n$, we have 
\begin{equation}
\mu(\mathbf{\mathbold{\phi}})\cdot f(m)=\sup_{1\leq i,j\leq n, i\neq j}
\frac{|\mathbf{{\Phi}}_i^\top\mathbf{\mathbold{\phi}}_j|}{m}\cdot f(m)\to 0
\end{equation}
almost surely as $m\to\infty$. Hence, with high probability, we have $\mu(\mathbf{{\Phi}})\cdot 
f(m)\leq 1$ for large $m$, and the result is proved.
\end{proof}

\begin{remark}
The proof we just did assumes that each entry of $\mathbf{{\Phi}}$ follows standard normal 
distribution $\mathcal{N}(0,1)$, however, it is not necessary to make such an assumption. The 
lemma will be true as long as each entry of $\mathbf{{\Phi}}$ are i.i.d. with mean zero and 
finite variance. The proof will be almost exactly the same as what we just did, except with 
modification of some constants, we will leave it to the interested readers.
\end{remark}
\noindent
Now we are able to develop our main results, which we summarize them in the following Theorem 
\ref{thm1} and Theorem \ref{thm2}. 
\begin{theorem} 
\label{thm1}
Suppose $\mu(\mathbf{{\Phi}})\leq \frac{1}{f(m)}$ for a function $f$ which satisfies 
$f(m)=o(\sqrt{m})$ and $f(m)\to\infty$ as $m\to \infty$. If the sparsity $s$ of the 
true signal $\mathbf{x}$ satisfies $2\leq s\leq\frac{f(m)}{5}$, 
then the following statement is true: For 
large $m$, among the two indices selected from the column indices of $\mathbf{{\Phi}}$ 
in the first iteration of Algorithm $3$, at least one index is the correct one.
\end{theorem}


\begin{proof}
Without loss of generality, let us assume each column $\mathbf{\mathbold{\phi}}_i$ is normalized, and assume the support 
set of $\mathbf{x}$ is $\Omega=\{1,2,\cdots,s\}$. Then 
$\mathbf{b}=\mathbf{{\Phi}}\mathbf{x}=x_1\mathbf{\mathbold{\phi}}_1+x_2\mathbf{\mathbold{\phi}}_2+\cdots+x_s
\mathbf{\mathbold{\phi}}_s
=\sum_{k=1}^s x_k \mathbf{\mathbold{\phi}}_k$, where $\mathbf{\mathbold{\phi}}_i$ is the $i$-column of 
$\mathbf{{\Phi}}$. \\
In the first iteration,
for each $1\leq i,j\leq n$, to minimize $\|\mathbf{\mathbold{\phi}}_iu+\mathbf{\mathbold{\phi}}_jv-\mathbf{b}\|_2$, 
it is equivalent to maximize the projection of $\mathbf{b}$ onto the hyperplane spanned by 
$\bfa_i$ and $\bfa_j$, which is
\begin{align*}
Proj(\mathbf{b})&=
\begin{bmatrix}\mathbf{\mathbold{\phi}}_i & \mathbf{\mathbold{\phi}}_j \end{bmatrix} 
\begin{bmatrix} \mathbf{\mathbold{\phi}}_i^\top\mathbf{\mathbold{\phi}}_i & \mathbf{\mathbold{\phi}}_i^\top \mathbf{\mathbold{\phi}}_j\\ \mathbf{\mathbold{\phi}}_j^\top \mathbf{\mathbold{\phi}}_i & \mathbf{\mathbold{\phi}}_j^\top \mathbf{\mathbold{\phi}}_j \end{bmatrix}^{-1}
\begin{bmatrix} \mathbf{\mathbold{\phi}}_i^\top \\ \mathbf{\mathbold{\phi}}_j^\top \end{bmatrix} 
\cdot\mathbf{b} \\
&=\begin{bmatrix}\mathbf{\mathbold{\phi}}_i & \mathbf{\mathbold{\phi}}_j \end{bmatrix} 
\begin{bmatrix} 1 & \mathbf{\mathbold{\phi}}_i^\top \mathbf{\mathbold{\phi}}_j\\ \mathbf{\mathbold{\phi}}_j^\top \mathbf{\mathbold{\phi}}_i & 1 \end{bmatrix}^{-1}
\begin{bmatrix} \mathbf{\mathbold{\phi}}_i^\top \\ \mathbf{\mathbold{\phi}}_j^\top \end{bmatrix}
\cdot(\sum_{k=1}^s x_k \mathbf{\mathbold{\phi}}_k)
\end{align*}
as each column $\mathbf{\mathbold{\phi}}_i$ is normalized.  We further have 
\begin{align*}
Proj(\mathbf{b})&=\frac{1}{1-|\mathbf{\mathbold{\phi}}_i^\top \mathbf{\mathbold{\phi}}_j|^2}
\begin{bmatrix} \mathbf{\mathbold{\phi}}_i & \mathbf{\mathbold{\phi}}_j \end{bmatrix} 
\begin{bmatrix} 1 & -\mathbf{\mathbold{\phi}}_i^\top \mathbf{\mathbold{\phi}}_j\\ 
-\mathbf{\mathbold{\phi}}_j^\top \mathbf{\mathbold{\phi}}_i & 1 
\end{bmatrix}
\begin{bmatrix} \mathbf{\mathbold{\phi}}_i^\top \\ \mathbf{\mathbold{\phi}}_j^\top \end{bmatrix} 
\cdot(\sum_{k=1}^s x_k \mathbf{\mathbold{\phi}}_k)\\
&=\frac{1}{1-|\mathbf{\mathbold{\phi}}_i^\top \mathbf{\mathbold{\phi}}_j|^2}(\mathbf{\mathbold{\phi}}_i \mathbf{\mathbold{\phi}}_i^\top+\mathbf{\mathbold{\phi}}_j 
\mathbf{\mathbold{\phi}}_j^\top-(\mathbf{\mathbold{\phi}}_i^\top \mathbf{\mathbold{\phi}}_j)\mathbf{\mathbold{\phi}}_j \mathbf{\mathbold{\phi}}_i^\top-(\mathbf{\mathbold{\phi}}_j^
\top \mathbf{\mathbold{\phi}}_i)\mathbf{\mathbold{\phi}}_i \mathbf{\mathbold{\phi}}_j^\top)\cdot(\sum_{k=1}^s x_k \mathbf{\mathbold{\phi}}_k).
\end{align*}
In order to show Theorem \ref{thm1}, we only need to show that $\|Proj(\mathbf{b})\|$ is not 
maximized when $i\notin S$ and $j\notin S$. We can do it by showing that $\|Proj(\mathbf{b})\|$ 
when both $i, j\notin S$ is strictly less than $\|Proj(\mathbf{b})\|$ when either $i\in S$ or 
$j\in S$. \\
Firstly, suppose both $i,j\notin S$. By applying triangle inequality together with the 
assumption $\mu(\mathbf{{\Phi}})\leq \frac{1}{f(m)}$, we get 
\begin{align*}
&\|Proj(\mathbf{b})_{i,j\notin S}\|\cr 
&=\|\frac{1}{1-|\mathbf{\mathbold{\phi}}_i^\top \mathbf{\mathbold{\phi}}_j|^2}(\mathbf{\mathbold{\phi}}_i \mathbf{\mathbold{\phi}}_i^\top+\mathbf{\mathbold{\phi}}_j 
\mathbf{\mathbold{\phi}}_j^\top-(\mathbf{\mathbold{\phi}}_i^\top \mathbf{\mathbold{\phi}}_j)\mathbf{\mathbold{\phi}}_j \mathbf{\mathbold{\phi}}_i^\top-(\mathbf{\mathbold{\phi}}_j^
\top \mathbf{\mathbold{\phi}}_i)\mathbf{\mathbold{\phi}}_i \mathbf{\mathbold{\phi}}_j^\top)\cdot(\sum_{k=1}^s x_k \mathbf{\mathbold{\phi}}_k)\| \\
& \leq \frac{1}{1-|\mathbf{\mathbold{\phi}}_i^\top \mathbf{\mathbold{\phi}}_j|^2}
\sum_{k=1}^s |x_k|((|\mathbf{\mathbold{\phi}}_i^
\top\mathbf{\mathbold{\phi}}_k|+|\mathbf{\mathbold{\phi}}_i^\top\mathbf{\mathbold{\phi}}_j\|\mathbf{\mathbold{\phi}}_j^\top\mathbf{\mathbold{\phi}}_k|)\cdot \|
\mathbf{\mathbold{\phi}}_i\|+(|\mathbf{\mathbold{\phi}}_j^\top\mathbf{\mathbold{\phi}}_k|+ |\mathbf{\mathbold{\phi}}_j^\top\mathbf{\mathbold{\phi}}_i|
|\mathbf{\mathbold{\phi}}_i^\top\mathbf{\mathbold{\phi}}_k|) \cdot \|\mathbf{\mathbold{\phi}}_j\|) \\
&=\frac{1}{1-|\mathbf{\mathbold{\phi}}_i^\top \mathbf{\mathbold{\phi}}_j|^2}\sum_{k=1}^s |x_k| (|\mathbf{\mathbold{\phi}}_i^\top\mathbf
{\mathbold{\phi}}_k|+|\mathbf{\mathbold{\phi}}_i^\top\mathbf{\mathbold{\phi}}_j| |\mathbf{\mathbold{\phi}}_j^\top\mathbf{\mathbold{\phi}}_k|
+|\mathbf{\mathbold{\phi}}_j^\top\mathbf{\mathbold{\phi}}_k|+ |\mathbf{\mathbold{\phi}}_j^\top\mathbf{\mathbold{\phi}}_i| 
|\mathbf{\mathbold{\phi}}_i^\top\mathbf{\mathbold{\phi}}_k|) \\
&\leq \frac{1}{1-1/f^2(m)}(1/f(m)+1/f^2(m)+1/f(m)+1/f^2(m))(\sum_{k=1}^s |x_k|)\\
&=\frac{f^2(m)}{f^2(m)-1}(\frac{2}{f(m)}+\frac{2}{f^2(m)})(\sum_{k=1}^s |x_k|) 
=\frac{2f(m)+2}{f^2(m)-1}\cdot(\sum_{k=1}^s |x_k|).
\end{align*}
Secondly, suppose $i\in S$ or $j\in S$. Without loss of generality, let us assume $i\in S$, $i=1$
 and $|x_1|=\max_{1\leq i\leq s}{|x_i|}$ is the one of the largest entries in absolute value. 
By applying triangle inequality together with the 
assumption $\mu(\mathbf{{\Phi}})\leq \frac{1}{f(m)}$, we get
\begin{align*}
& \|Proj(\mathbf{b})_{i\in S}\| \cr
&=\|\frac{1}{1-|\mathbf{{\Phi}}_1^\top \mathbf{\mathbold{\phi}}_j|^2}(\mathbf{\mathbold{\phi}}_1 
\mathbf{\mathbold{\phi}}_1^\top+\mathbf{\mathbold{\phi}}_j \mathbf{\mathbold{\phi}}_j^\top-(\mathbf{\mathbold{\phi}}_1^\top \mathbf{\mathbold{\phi}}_j)\mathbf{\mathbold{\phi}}_j 
\mathbf{\mathbold{\phi}}_1^\top-(\mathbf{\mathbold{\phi}}_j^\top \mathbf{\mathbold{\phi}}_1)\mathbf{\mathbold{\phi}}_1 \mathbf{\mathbold{\phi}}_j^\top)
\cdot(\sum_{k=1}^s x_k \mathbf{\mathbold{\phi}}_k)\| \\
&\geq \frac{1}{1-|\mathbf{\mathbold{\phi}}_1^\top \mathbf{\mathbold{\phi}}_j|^2}(|x_1|(\|\mathbf{\mathbold{\phi}}_1\|-|\mathbf{\mathbold{\phi}}_j^
\top\mathbf{\mathbold{\phi}}_1|\cdot\|\mathbf{\mathbold{\phi}}_j\|-|\mathbf{\mathbold{\phi}}_1^\top\mathbf{\mathbold{\phi}}_j|\cdot\|\mathbf{\mathbold{\phi}}_j\|-|\mathbf{\mathbold{\phi}}_1^
\top\mathbf{\mathbold{\phi}}_j||\mathbf{\mathbold{\phi}}_j^\top\mathbf{\mathbold{\phi}}_1|\cdot\|\mathbf{\mathbold{\phi}}_1\|)) \\
&-\frac{1}{1-|\mathbf{\mathbold{\phi}}_1^\top \mathbf{\mathbold{\phi}}_j|^2}(2/f(m)+2/f^2(m))\cdot(\sum_{k=2}^s|x_k|) \\
&\geq\frac{1}{1-|\mathbf{\mathbold{\phi}}_1^\top \mathbf{\mathbold{\phi}}_j|^2}[(1-2/f(m)-1/f^2(m))|x_1|-(2/f(m)+2/f^2(m))
\cdot(\sum_{k=2}^s |x_k|)].
\end{align*}
It follows that to show $\|Proj(\mathbf{b})_{i\in S}\|\geq \|Proj(\mathbf{b})_{i,j\notin S}\|$, 
it is equivalent to show
\begin{align} 
\label{eqn1}
\frac{1}{1-1/m^2}((1-2/f(m)-1/f^2(m))|x_1|-(2/f(m) &+2/f^2(m))\cdot(\sum_{k=2}^s |x_k|)) 
\cr
\geq & \frac{2f(m)+2}{f^2(m)-1}\cdot\sum_{k=1}^s |x_k|.
\end{align}
Since $|x_1|=\max_{1\leq i\leq s}|x_i|$, it suffices to show
\begin{align*}
\frac{1}{1-1/m^2}((1-2/f(m)-1/f^2(m))-(2/f(m)+2/f^2(m))\cdot(s-1)) \geq \frac{2f(m)+2}{f^2(m)-1}\cdot s,
\end{align*}
which is equivalent to 
\begin{equation} 
\label{eqn2}
1-\frac{2s}{f(m)}-\frac{2s-1}{f^2(m)}\geq (1-\frac{1}{m^2})\cdot
\frac{2f(m)+2}{f^2(m)-1}\cdot s.
\end{equation}
Since $s\leq\frac{f(m)}{5}$ and $f(m)\to\infty$ as $m\to\infty$, we have for the left 
side $1-\frac{2s}{f(m)}-\frac{2s-1}{f^2(m)}>1/2$ for large $m$. For the right side, we have
$(1-\frac{1}{m^2})\cdot\frac{2f(m)+2}{f^2(m)-1}\cdot s<1/2$ for large $m$. Hence there 
are certain threshold $m_0$ such that (\ref{eqn2}) hold as long as $m\geq m_0$. Therefore, the theorem is true. 
\end{proof}

\begin{theorem} 
\label{thm2}
Under the same condition as Theorem~\ref{thm1}, the exact recovery of the s-sparse signal 
$\mathbf{x}$ can be guaranteed in $s$ iterations by using Algorithm $3$.
\end{theorem}
\begin{proof}
By Theorem~\ref{thm1} or from its proof, we know that 
the first iteration will pick at least one correct column index. Without loss of generality, suppose the first correct index that is picked in the first iteration is the first column, and the other column which is picked together with the first column is the $j$-th column. 
Then in the second iteration, the residual vector gets updated to $\mathbf{r}_1=\mathbf{b}-x_1\mathbold{\phi}_1-x_j\mathbold{\phi}_j$, where $(x_1,x_j)'=\mathbf{{\Phi}}_{\mathcal{S}_1}^{\dagger}\mathbf{r}_1$. The matrix $\mathbf{{\Phi}}$ gets updated to $\mathbf{{\Psi}}$ where $\mathbf{{\Psi}}$ is the matrix $\mathbf{{\Phi}}$ but with either the first column or the first and $j$-th column being replaced by $\mathbf{0}$ vectors because of the update step 
$\mathbf{\Psi}_{\{i_k,j_k\}}=\mathbf{0}$ in Algorithm $3$. By the same analysis, we can conclude that the second iteration will also pick at least one correct column index which is different those being picked in the first iteration.
Thus each iteration will pick at least one correct column index which are different 
from what are picked from previous iterations, and hence the support set $S$ is recovered within $s$ total iterations. 
\end{proof}

\subsection{Signal Recovery in the Noisy Setting} 
Similar to the noiseless case above, we now obtain a sufficient condition for the recovery 
of $s$-sparse signal with Algorithm $3$ from $\mathbf{b}=\mathbf{{\Phi}}\mathbf{x}+\mathbf{v}$ 
with noise vector $\mathbf{v}$.

\begin{theorem} 
\label{thm3}
Suppose the noise vector $\mathbf{v}$ satisfies $\|\mathbf{v}\|\leq \epsilon$ and 
$\mu(\mathbf{\Phi})\leq \frac{1}{f(m)}$ for some function $f$ which 
satisfies $f(m)=o(\sqrt{m})$ and $f(m)\to\infty$ as $m\to \infty$. If the sparsity $s$ of the 
true signal $\mathbf{x}$ satisfies $2\leq s\leq \frac{f(m)}{5}$ and suppose 
$\min_{x_i\neq 0}|x_i|>\frac{f(m)-5s}{5f(m)}\cdot\epsilon$. Then the following statement 
is true: For 
large $m$, among the two indices selected from the column indices of $\mathbf{{\Phi}}$ 
in each iteration of Algorithm $3$, at least one index is the correct one, and hence 
the exact recovery of the s-sparse signal $\mathbf{x}$ can be guaranteed in $s$ iterations 
by using Algorithm $3$ in the noise case.
\end{theorem}
\begin{proof}
The analysis is similar to Theorem~\ref{thm1}. 
Without loss of generality, let us assume each column $\mathbf{\mathbold{\phi}}_i$ is normalized, and assume the support 
set of $\mathbf{x}$ is $\Omega=\{1,2,\cdots,s\}$. Then
$\mathbf{b}=\mathbf{b}_0+\mathbf{v}
=\mathbold{\phi}\mathbf{x}+\mathbf{v}=x_1\mathbf{\mathbold{\phi}}_1+x_2\mathbf{\mathbold{\phi}}_2
+\cdots+x_s\mathbf{\mathbold{\phi}}_s+\mathbf{v}
=\sum_{k=1}^s x_k \mathbf{{\Phi}}_k+\mathbf{v}$. Let us firstly consider the first iteration.  \\
The projection of $\mathbf{b}$ onto the hyperplane 
spanned by $\mathbf{\mathbold{\phi}}_i$ and $\mathbf{\mathbold{\phi}}_j$ is
\begin{align*}
Proj(\mathbf{b})&=Proj(\mathbf{b}_0+\mathbf{v}) \\
&=Proj(\mathbf{b}_0)+
\begin{bmatrix}\mathbf{\mathbold{\phi}}_i & \mathbf{\mathbold{\phi}}_j \end{bmatrix} 
\begin{bmatrix} 1 & \mathbf{\mathbold{\phi}}_i^\top \mathbf{\mathbold{\phi}}_j\\ \mathbf{\mathbold{\phi}}_j^\top \mathbf{\mathbold{\phi}}_i & 1 \end{bmatrix}^{-1}
\begin{bmatrix} \mathbf{\mathbold{\phi}}_i^\top \\ \mathbf{\mathbold{\phi}}_j^\top \end{bmatrix}
\cdot\mathbf{v} \\
&=Proj(\mathbf{b}_0)+\frac{1}{1-|\mathbf{\mathbold{\phi}}_i^\top \mathbf{\mathbold{\phi}}_j|^2}(\mathbf{\mathbold{\phi}}_i \mathbf{\mathbold{\phi}}_i^\top+\mathbf{\mathbold{\phi}}_j 
\mathbf{\mathbold{\phi}}_j^\top-(\mathbf{\mathbold{\phi}}_i^\top \mathbf{\mathbold{\phi}}_j)\mathbf{\mathbold{\phi}}_j \mathbf{\mathbold{\phi}}_i^\top-(\mathbf{\mathbold{\phi}}_j^
\top \mathbf{\mathbold{\phi}}_i)\mathbf{\mathbold{\phi}}_i \mathbf{\mathbold{\phi}}_j^\top)\cdot\mathbf{v}
\end{align*}
In order to show Theorem~\ref{thm3}, we only need to show that $\|Proj(\mathbf{b})\|$ is not 
maximized when $i\notin S$ and $j\notin S$. We can do it by showing that $\|Proj(\mathbf{b})\|$ 
when both $i, j\notin S$ is strictly less than $\|Proj(\mathbf{b})\|$ when either $i\in S$ or 
$j\in S$. \\
Firstly, suppose both $i,j\notin S$. By applying triangle inequality together with the 
assumption $\mu(\mathbf{\Phi})\leq \frac{1}{f(m)}$, we get 
\begin{align*}
&\|Proj(\mathbf{b})_{i,j\notin S}\|\cr 
&=\|Proj(\mathbf{b}_0)+\frac{1}{1-|\mathbf{\mathbold{\phi}}_i^\top \mathbf{\mathbold{\phi}}_j|^2}(\mathbf{\mathbold{\phi}}_i \mathbf{\mathbold{\phi}}_i^\top+\mathbf{\mathbold{\phi}}_j 
\mathbf{\mathbold{\phi}}_j^\top-(\mathbf{\mathbold{\phi}}_i^\top \mathbf{\mathbold{\phi}}_j)\mathbf{\mathbold{\phi}}_j \mathbf{\mathbold{\phi}}_i^\top-(\mathbf{\mathbold{\phi}}_j^
\top \mathbf{\mathbold{\phi}}_i)\mathbf{\mathbold{\phi}}_i \mathbf{\mathbold{\phi}}_j^\top)\cdot\mathbf{v}\|  \\
& \leq \|Proj(\mathbf{b}_0)\| 
+\frac{1}{1-|\mathbf{\mathbold{\phi}}_i^\top \mathbf{\mathbold{\phi}}_j|^2}((|
\mathbold{\phi}_i^\top\mathbf{v}|+|\mathbold{\phi}_j^\top \mathbold{\phi}_i||\mathbold{\phi}_j^\top\mathbf{v}|)\|\mathbold{\phi}_i\|+(|
\mathbold{\phi}_j^\top\mathbf{v}|+|\mathbold{\phi}_i^\top \mathbold{\phi}_j||\mathbold{\phi}_i^\top\mathbf{v}|)\|\mathbold{\phi}_j\|) \\
&=\|Proj(\mathbf{b}_0)\|  
+\frac{1}{1-|\mathbf{\mathbold{\phi}}_i^\top \mathbf{\mathbold{\phi}}_j|^2}(|
\mathbold{\phi}_i^\top\mathbf{v}|+|\mathbold{\phi}_j^\top \mathbold{\phi}_i||\mathbold{\phi}_j^\top\mathbf{v}|+|
\mathbold{\phi}_j^\top\mathbf{v}|+|\mathbold{\phi}_i^\top \mathbold{\phi}_j||\mathbold{\phi}_i^\top\mathbf{v}|)   \\
&\leq \frac{2f(m)+2}{f^2(m)-1}\cdot(\sum_{k=1}^s |x_k|) 
+ \frac{1}{1-1/f^2(m)}(\|\mathbf{v}\|+\frac{1}{f(m)}\cdot\|\mathbf{v}\|+\|\mathbf{v}\|+\frac{1}{f(m)}\cdot\|\mathbf{v}\|)  \\
&\leq\frac{2f(m)+2}{f^2(m)-1}\cdot(\sum_{k=1}^s |x_k|)+\frac{2f(m)(f(m)+1)}{f^2(m)-1}\cdot\epsilon.
\end{align*}
Secondly, suppose $i\in S$ or $j\in S$. Without loss of generality, let us assume $i\in S$, $i=1$ and $|x_1|=\max_{1\leq i\leq s}{|x_i|}$ is the one of the largest entries in absolute value. By applying triangle inequality together with the 
assumption $\mu(\mathbf{{\Phi}})\leq \frac{1}{f(m)}$, we get
\begin{align*}
& \|Proj(\mathbf{b})_{i\in S}\| \cr
&=\|Proj(\mathbf{b}_0)+\frac{1}{1-|\mathbf{\mathbold{\phi}}_1^\top \mathbf{\mathbold{\phi}}_j|^2}(\mathbf{\mathbold{\phi}}_1 
\mathbf{\mathbold{\phi}}_1^\top+\mathbf{\mathbold{\phi}}_j \mathbf{\mathbold{\phi}}_j^\top-(\mathbf{\mathbold{\phi}}_1^\top \mathbf{\mathbold{\phi}}_j)\mathbf{\mathbold{\phi}}_j 
\mathbf{\mathbold{\phi}}_1^\top-(\mathbf{\mathbold{\phi}}_j^\top \mathbf{\mathbold{\phi}}_1)\mathbf{\mathbold{\phi}}_1 \mathbf{\mathbold{\phi}}_j^\top)
\cdot\mathbf{v}\| \\
&\geq \|Proj(\mathbf{b}_0)\|-\|\frac{1}{1-|\mathbf{\mathbold{\phi}}_1^\top \mathbf{\mathbold{\phi}}_j|^2}(\mathbf{\mathbold{\phi}}_1 
\mathbf{\mathbold{\phi}}_1^\top+\mathbf{\mathbold{\phi}}_j \mathbf{\mathbold{\phi}}_j^\top-(\mathbf{\mathbold{\phi}}_1^\top \mathbf{\mathbold{\phi}}_j)\mathbf{\mathbold{\phi}}_j 
\mathbf{\mathbold{\phi}}_1^\top-(\mathbf{\mathbold{\phi}}_j^\top \mathbf{\mathbold{\phi}}_1)\mathbf{\mathbold{\phi}}_1 \mathbf{\mathbold{\phi}}_j^\top)
\cdot\mathbf{v}\| \\
&\geq (1-2/f(m)-1/f^2(m))|x_1|-(2/f(m)+2/f^2(m))
\cdot(\sum_{k=2}^s |x_k|) \\
&-\frac{1}{1-|\mathbf{\mathbold{\phi}}_1^\top \mathbf{\mathbold{\phi}}_j|^2}(\|\mathbf{v}\|+\frac{1}{f(m)}\cdot\|\mathbf{v}\|+\|\mathbf{v}\|+\frac{1}{f(m)}\cdot\|\mathbf{v}\|)   \\
&\geq (1-2/f(m)-1/f^2(m))|x_1|-(2/f(m)+2/f^2(m))
\cdot(\sum_{k=2}^s |x_k|)-(2\epsilon+2\epsilon/f(m)).
\end{align*}
It remains to show $\|Proj(\mathbf{b})_{i\in S}\|\geq \|Proj(\mathbf{b})_{i,j\notin S}\|$, 
which is equivalent to show 
\begin{align*}
&(1-2/f(m)-1/f^2(m))|x_1|-(2/f(m)+2/f^2(m))
\cdot(\sum_{k=2}^s |x_k|)-(2\epsilon+2\epsilon/f(m)) \\
&\geq \frac{2f(m)+2}{f^2(m)-1}\cdot(\sum_{k=1}^s |x_k|)+\frac{2f(m)(f(m)+1)}{f^2(m)-1}\cdot\epsilon
\end{align*}
Since $|x_1|=\max_{1\leq i\leq s}|x_i|$, it suffices to show
\begin{align*}
&1-\frac{2}{f(m)}-\frac{1}{f^2(m)}-(\frac{2}{f(m)}+\frac{2}{f^2(m)})\cdot(s-1)-\frac{2\epsilon}{|x_1|}-\frac{2\epsilon}{f(m)\cdot|x_1|} \cr 
&\geq \frac{2f(m)+2}{f^2(m)-1}\cdot s+\frac{2f(m)(f(m)+1)}{f^2(m)-1}\cdot\frac{\epsilon}{|x_1|},
\end{align*}
which is equivalent to
\begin{equation}  \label{eqn3}
1-\frac{2s}{f(m)}-\frac{2s-1}{f^2(m)}\geq 
\frac{2f(m)+2}{f^2(m)-1}\cdot s+\frac{2\epsilon}{|x_1|}+\frac{2\epsilon}{f(m)\cdot|x_1|}+\frac{2f(m)(f(m)+1)}{f^2(m)-1}\cdot\frac{\epsilon}{|x_1|}.
\end{equation} 
Since $s\leq \frac{f(m)}{5}$ and $f(m)\to\infty$ as $m\to\infty$, we have 
for the left-hand 
side $1-\frac{2s}{f(m)}-\frac{2s-1}{f^2(m)}>1/2$ for large $m$. 
For the right-hand side, we have 
$\frac{2f(m)+2}{f^2(m)-1}\cdot s<1/2$ for large $m$ and 
$\frac{2\epsilon}{f(m)\cdot|x_1|}\to 0$ as $m\to\infty$. By assumption 
$\min_{x_i\neq 0}|x_i|>\frac{f(m)-5s}{5f(m)}\cdot\epsilon$, we have 
$\frac{2\epsilon}{|x_1|}+\frac{2f(m)(f(m)+1)}{f^2(m)-1}\cdot\frac{\epsilon}
{|x_1|}<1/2$ for large $m$.
Hence there are certain threshold $m_0$ such that (\ref{eqn3}) hold as long as $m\geq m_0$.  
\\
For the subsequent iterations, we have  
$\frac{2\epsilon}{|x_i|}+\frac{2f(m)(f(m)+c)}{f^2(m)-c^2}\cdot
\frac{\epsilon}{|x_i|}<1/2$ for 
all $i\in\{1,\cdots, s\}$ when $m$ is large. Therefore $(\ref{eqn3})$ 
will hold for each 
subsequent iteration and hence each iteration will select at least one 
correct column index. 
\end{proof}

\subsection{The Convergence Rate of Algorithm 3.}
Let us continue to study the convergence of Algorithm~2.  It is clear that 
the $k$th residual vector  
$\mathbf{r}_k= \mathbf{b}- \mathbf{Proj}_{\mathcal{S}_k}(\mathbf{b})$ and 
$$
\mathbf{r}_{k}= \mathbf{r}_{k-1}- \mathbf{Proj}_{(i_k,j_k)}(\mathbf{r}_{k-1})
$$
from Algorithm~2,    
where we have used  
$\mathbf{Proj}_{(i_k,j_k)}(\mathbf{r}_{k-1})$ to denote the projection of 
$\mathbf{r}_{k-1}$ 
onto the hyperplane spanned by $\mathbold{\phi}_{i_k}$ and $\mathbold{\phi}_{j_k}$. 
Notice that we can rewrite the above equality as follows
\begin{equation} \label{eqnresidual}
\mathbf{r}_{k-1}=\mathbf{r}_k+\mathbf{Proj}_{(i_k,j_k)}(\mathbf{r}_{k-1})   
\end{equation}
and note that $\mathbf{r}_{k}$ is orthogonal to 
$\mathbf{Proj}_{\{i_k,j_k\}}(\mathbf{r}_{k-1})$. By squaring both sides of 
equation (\ref{eqnresidual}), we have
\begin{equation}
\label{decrease}
\|\mathbf{r}_{k-1}\|^2 = \|\mathbf{r}_k\|^2 + 
\|\mathbf{Proj}_{\{i_k,j_k\}}(\mathbf{r}_{k-1})\|^2
\end{equation}
or $\|\mathbf{r}_k\| \le \|\mathbf{r}_{k-1}\|$. That is, the residual 
vectors $\mathbf{r}_k$ is decreasing. 
\\

\noindent
In fact, we can establish the rate of convergence of the residual 
$\|\mathbf{r}_k\|$ if the sensing matrix $\mathbf{{\Phi}}$ is Gaussian. 
The following result explains that, for large $m$, the residual $\|{\bf r}_k\|$ of Algorithm $3$
decreases to 0 linearly with high probability, provided there exists a
sparse solution $\bfx$ such that $\mathbf{{\Phi}}\bfx = \bfb$.  

\begin{theorem}
\label{lemma-descrease}
Suppose $n\geq 2m$, let $\mathbf{{\Phi}}\in\mathbb{R}^{m\times n}$ be a Gaussian sensing matrix. 
Suppose that sparse signal $x$ can be exactly recovered within $K$ $(\leq \frac{m}{2})$ iterations by using Algorithm $3$, 
and suppose further that
the RIP constant $\delta_2\in (0, 1)$. Then for all $1\leq k \leq K$, there exists a constant $\alpha\in (0, 1)$ such that 
\begin{equation}
\label{rate}
\|\mathbf{r}_k\|^2 \le \alpha \|\mathbf{r}_{k-1}\|^2
\end{equation}
for large $m$ with high probability . 
\end{theorem}

\begin{proof}
First of all, we know $\mathbf{Proj}_{\{i_k,j_k\}}(\mathbf{r}_{k-1})
=\mathbf{\Phi}_{i_k,j_k}(\mathbf{\Phi}_{i_k,j_k}^\dagger \mathbf{r}_{k-1})$. For convenience, 
let $\bfx_k = \mathbf{\Phi}_{i_k,j_k}^\dagger \mathbf{r}_{k-1}$. Note that $\bfx\in 
\mathbb{R}^n$, but $\|\bfx_k\|_0= 2$.
From (\ref{decrease}), 
by using the RIP property with $s=2$, we obtain
\begin{align} \label{eqnnormresidual}
\|\mathbf{r}_{k}\|^2&=\|\mathbf{r}_{k-1}\|^2-\|\mathbf{Proj}_{\{i_k,j_k\}}(\mathbf{r}_{k-1})\|^2\cr
&=\|\mathbf{r}_{k-1}\|^2-\|\mathbf{{\Phi}}_{\{i_k,j_k\}}\mathbf{x}_k\|^2 
\leq \|\mathbf{r}_{k-1}\|^2-(1-\delta_2)\|\mathbf{x}_k\|^2.  
\end{align}
\noindent
Since $\mathbf{x}_k$ is the solution to (\ref{Shensidea}), by assuming each 
column of $\mathbf{{\Phi}}$ is normalized, we have 
\begin{align*}
\mathbf{x}_k&=\mathbf{{\Phi}}_{\{i_k,j_k\}}^{\dagger}\cdot\mathbf{Proj}_{
\{i_k,j_k\}}(\mathbf{r}_{k-1})=\mathbf{{\Phi}}_{\{i_k,j_k\}}^{\dagger}
\cdot\mathbf{{\Phi}}_{\{i_k,j_k\}} \mathbf{{\Phi}}_{\{i_k,j_k\}}^{\dagger}
\cdot\mathbf{r}_{k-1}=\mathbf{{\Phi}}_{\{i_k,j_k\}}^{\dagger}\cdot
\mathbf{r}_{k-1}\\
&=
\begin{bmatrix} 1 & \mathbf{\mathbold{\phi}}_{i_k}^\top \mathbf{\mathbold{\phi}}_{j_k}\\ 
\mathbf{\mathbold{\phi}}_{j_k}^\top \mathbf{\mathbold{\phi}}_{i_k} & 1 \end{bmatrix}^{-1}
\begin{bmatrix} \mathbf{\mathbold{\phi}}_{i_k}^\top \\ \mathbf{\mathbold{\phi}}_{j_k}^\top \end{bmatrix}
\cdot\mathbf{r}_{k-1} \\
&=\frac{1}{1-|\mathbold{\phi}_{i_k}^\top \mathbold{\phi}_{j_k}|^2}
\begin{bmatrix} 1 & -\mathbf{\mathbold{\phi}}_{i_k}^\top \mathbf{\mathbold{\phi}}_{j_k}\\ 
-\mathbf{\mathbold{\phi}}_{j_k}^\top \mathbf{\mathbold{\phi}}_{i_k} & 1 \end{bmatrix}
\begin{bmatrix} \mathbf{\mathbold{\phi}}_{i_k}^\top \\ \mathbf{\mathbold{\phi}}_{j_k}^\top \end{bmatrix}
\cdot\mathbf{r}_{k-1} \\
&=\frac{1}{1-|\mathbold{\phi}_{i_k}^\top \mathbold{\phi}_{j_k}|^2}
\begin{bmatrix}
\mathbold{\phi}_{i_k}^\top-\mathbold{\phi}_{i_k}^\top\mathbold{\phi}_{j_k}\mathbold{\phi}_{j_k}^\top \\
\mathbold{\phi}_{j_k}^\top-\mathbold{\phi}_{j_k}^\top\mathbold{\phi}_{i_k}\mathbold{\phi}_{i_k}^\top 
\end{bmatrix}
\cdot\mathbf{r}_{k-1} 
=\frac{1}{1-|\mathbold{\phi}_{i_k}^\top \mathbold{\phi}_{j_k}|^2}
\begin{bmatrix}
\mathbold{\phi}_{i_k}^\top(I-\mathbold{\phi}_{j_k}\mathbold{\phi}_{j_k}^\top) \\
\mathbold{\phi}_{j_k}^\top(I-\mathbold{\phi}_{i_k}\mathbold{\phi}_{i_k}^\top)
\end{bmatrix}
\cdot\mathbf{r}_{k-1}.
\end{align*}
\noindent
Since matrix $\mathbf{{\Phi}}$ is Gaussian, and since each column $\mathbold{\phi}_i$ is normalized, 
we have for any $\epsilon>0$. there is $m_0$ such that $\|\mathbold{\phi}_{i_k}\mathbold{\phi}_{j_k}^\top\|_{\infty}\leq \epsilon$ 
for all $m\geq m_0$ with high probability. Therefore, we have 

\begin{equation} 
\label{normestimate}
\|{\bf x}_k\|^2  \geq 
\|
\begin{bmatrix}
\mathbold{\phi}_{i_k}^\top(I-\mathbold{\phi}_{j_k}\mathbold{\phi}_{j_k}^\top) \\
\mathbold{\phi}_{j_k}^\top(I-\mathbold{\phi}_{i_k}\mathbold{\phi}_{i_k}^\top)
\end{bmatrix}
\cdot\mathbf{r}_{k-1}\|^2\geq \frac{1}{2}\cdot\|
\begin{bmatrix}
\mathbold{\phi}_{i_k}^\top \\
\mathbold{\phi}_{j_k}^\top
\end{bmatrix}
\cdot\mathbf{r}_{k-1}\|^2 
=\frac{1}{2}\cdot(|\mathbold{\phi}_{i_k}^\top \mathbf{r}_{k-1}|^2
+|\mathbold{\phi}_{j_k}^\top \mathbf{r}_{k-1}|^2) 
\end{equation}
for large $m$ with high probability. \\
\\
Now let us give an estimate of the right hand side of (\ref{normestimate}). Let $\mathbold{\phi}_p$ be the projection of $\mathbf{r}_{k-1}$ onto the 
hyperplane spanned by $\mathbold{\phi}_{i_k}$ 
and $\mathbold{\phi}_{j_k}$.  Letting the angle between $\mathbold{\phi}_{i_k}$ and $\mathbold{\phi}_{j_k}$ 
be $\theta$,  at 
least one of the quantity $|\mathbold{\phi}_{i_k}^\top \mathbf{r}_{k-1}|$ and 
$|\mathbold{\phi}_{j_k}^\top 
\mathbf{r}_{k-1}|$ are greater or equal to $|\mathbold{\phi}_{p}^\top 
\mathbf{r}_{k-1}|\cdot|\cos{\frac{\theta}{2}}|$. 
Without loss of generality, we can assume 
$|\mathbold{\phi}_{i_k}^\top \mathbf{r}_{k-1}|\geq |\mathbold{\phi}_{p}^\top 
\mathbf{r}_{k-1}|\cdot|\cos{\frac{\theta}{2}}|$. \\
Let $\mathbold{\phi}_{i_{max}}=\{\mathbold{\phi}_i: \max_{i}{|\mathbold{\phi}_{i}^\top\mathbf{r}_{k-1}|}\}$. 
We claim that we will have 
$|\mathbold{\phi}_{p}^\top \mathbf{r}_{k-1}|\geq |\mathbold{\phi}_{i_{max}}^\top \mathbf{r}_{k-1}|$. 
Otherwise $|\mathbold{\phi}_{i_{max}}^\top \mathbf{r}_{k-1}|>|\mathbold{\phi}_{p}^\top 
\mathbf{r}_{k-1}|$. Then we would have 
\begin{equation*}
\|\mathbf{Proj}_{(i_k,i_{max})}(\mathbf{r}_{k-1})\|\geq 
|\mathbold{\phi}_{i_{max}}^\top\mathbf{r}_{k-1}|>
|\mathbold{\phi}_{p}^\top\mathbf{r}_{k-1}|=\|\mathbf{Proj}_{(i_k,j_k)}(\mathbf{r}_{k-1})\|,
\end{equation*}
which contradicts the choice of the pair $(i_k,j_k)$. \\
\\
Notice that for $n>m$, we have $\mu({\bf \mathbf{\Phi}})=\max_{1\leq i,j\leq m, i
\neq j}|\mathbf{\mathbold{\phi}_i}^\top\mathbf{\mathbold{\phi}_j}|\geq \sqrt{\frac{n-m}{m(n-1)}}$ 
(see the previous subsection and also \cite{lowerboundmutualcoherence}). For the k-th iteration, the submatrix $\mathbf{{\Phi}}_{\mathcal{S}^c_{k-1}}$ is of size $m\times (n-2k+2)$. If $k\leq K$, then $\mu(\mathbf{{\Phi}}_{\mathcal{S}^c_{k-1}})\geq \sqrt{\frac{n-2k-m+2}{m(n-2k-1+2)}}\geq \sqrt{\frac{n-2K-m+2}{m(n-2K-1+2)}}$. Also note that by Lemma \ref{lem2}, we have $\mu(\mathbf{{\Phi}}_{\mathcal{S}^c_{k-1}})\leq \frac{1}{f(m)}$ for large $m$ with high probability, where $f(m)\to\infty$ as $m\to\infty$. Hence, for large $m$, we have
\begin{align*}
\frac{|\mathbold{\phi}_{i_{max}}^\top\mathbf{r}_{k-1}|^2}{ \|\mathbf{r}_{k-1}\|^2}=\mu^2([\mathbf{{\Phi}}_{\mathcal{S}^c_{k-1}},\mathbf{r}_{k-1}])\geq\mu^2(\mathbf{{\Phi}}_{\mathcal{S}^c_{k-1}})
\geq \frac{n-2K-m+2}{m(n-2K+1)}   
\end{align*}
with high probability. \\
\\
By combining the inequalities together, we have for large $m$
\begin{align} \label{chaininequality}
\|\mathbf{x}_k\|^2 &\geq\frac{1}{2}\cdot(|\mathbold{\phi}_{i_k}^\top \mathbf{r}_{k-1}|^2+|\mathbold{\phi}_{j_k}^\top \mathbf{r}_{k-1}|^2) \cr
&\geq \frac{1}{2}\cdot|\mathbold{\phi}_{i_k}^\top\mathbf{r}_{k-1}|^2  \cr
&\geq\frac{1}{2}\cdot|\mathbold{\phi}_{p}^\top\mathbf{r}_{k-1}|^2\cdot|\cos{\frac{\theta}{2}}|^2  \cr
&\geq\frac{1}{2}\cdot|\mathbold{\phi}_{i_{\max}}^\top\mathbf{r}_{k-1}|^2\cdot|\cos{\frac{\theta}{2}}|^2  \cr
&\geq\frac{1}{2}\cdot\frac{n-2K-m+2}{m(n-2K+1)}\cdot\|\mathbf{r}_{k-1}\|^2\cdot\frac{1+\cos{\theta}}{2}.  \cr
&\geq\frac{n-2K-m+2}{4m(n-2K+1)}\cdot\|\mathbf{r}_{k-1}\|^2\cdot(1-\mu).
\end{align}
with high probability. The last inequality holds because $|\cos{\theta}|\leq \mu$ and the half angle formula.  \\
By plugging (\ref{chaininequality}) back into (\ref{eqnnormresidual}), we have for $1\leq k\leq K$,
\begin{align}
\|\mathbf{r}_k\|^2 \leq \|\mathbf{r}_{k-1}\|^2-(1-\delta_2)\|\mathbf{x}_k\|^2 &\leq \|\mathbf{r}_{k-1}\|^2-(1-\delta_2)\cdot(1-\mu)\cdot\frac{n-2K-m+2}{4m(n-2K+1)}\|\mathbf{r}_{k-1}\|^2 \cr
&=(1-\frac{(1-\delta_2)(1-\mu)(n-2K-m+2)}{4m(n-2K+1)})\|\mathbf{r}_{k-1}\|^2.
\end{align}
for large $m$ with high probability.  \\
\\
Note that $\alpha<1$. Indeed, since the number of iterations will always be less than or equal to $m/2$ 
as the algorithm chooses two columns in every iteration and 
does not re-pick the same columns already  chosen from the previous iterations and since $n\geq 2m$, 
we have $n-2K-m+2>0$, i.e. $\alpha<1$. Hence, we choose this $\alpha$ to finish 
the proof.  
\end{proof}


\section{Computational Complexity and Numerical Results}
\subsection{Computational Complexity of Algorithm~3}
The total computational complexity of QOMP algorithm is dominated by the complexity in the 
iteration steps. In each iteration, the step 
$\min_{u,v\in\mathbb{R}}\{\|\mathbold{\phi}_i u+\mathbold{\phi}_j v-\mathbf{r_{k-1}}\|\}$ 
requires $O(m)$ operations, and finding the minimum of 
$\min_{u,v\in\mathbb{R}}\{\|\mathbold{\phi}_i u+\mathbold{\phi}_j v-\mathbf{r_{k-1}}\|\}$ 
while $i,j$ runs through $1$ to $n$ requires $O(n^2)$ operations, 
and since the iteration runs from $1$ to $s$, we have the total complexity is approximately 
\begin{align*}
O(m)\cdot O(n^2)\cdot s=O(m n^2 s),
\end{align*}
while the standard OMP algorithm has the complexity around $O(mns)$. \\
\\
However, when computing the 
$\min_{u,v\in\mathbb{R}}\{\|\mathbold{\phi}_i u+\mathbold{\phi}_j v-\mathbf{r_{k-1}}\|\}$, 
each of those 
minimizations is independent of one another 
when $i,j$ runs through $1$ to $n$, so we can use GPU or
parallel computing to improve the efficiency largely. 
In such a case, 
we can compute all the $\binom{n}{2}$ pairs of 
$\min_{u,v\in\mathbb{R}}\{\|\mathbold{\phi}_i u+\mathbold{\phi}_j v-\mathbf{r_{k-1}}\|\}$ simultaneously, which will reduce the total computational complexity to 
\begin{align*}
(O(m)+O(n^2))\cdot s=O(n^2 s).
\end{align*}
In the case that $m$ and $n$ are in the same scale, e.g. $n=2m$, the 
complexity is approximately $O(n^2 s)\approx O(mns)$, which is the same as OMP algorithm, and 
that is what we desired.

\subsection{Experimental Results}
As mentioned in the introduction, suppose the signal sparsity is $s$, then OMP algorithm can recover the signal within $s$ iterations if the restricted isometry constant of the sensing matrix satisfies certain condition. To investigate the performance of QOMP, we first compare the performance between QOMP and the standard OMP both within $s$ iterations for sensing matrix of size $32\times 128$. The frequency of exact recovery of 
each sparsity is computed by each method based on 1000 repetitions of solving a Gaussian random matrix of size $32\times 128$. We 
only show the results for $2\leq s<0.4m$ $(\approx 12)$ because the exact recovery rate is very low for both algorithm if $s\geq 
0.4m$, and to find an improvement for this range of $s$ is beyond the scope of this paper. As shown in Figure~\ref{fig1}, the QOMP 
algorithm has a much better than the standard OMP. 

\begin{figure}[htpb]
	\centering 
	\caption{Frequencies of exact reconstruction of signal within $s$ iterations for underdetermined linear systems of sizes $m\times n$ with $n=4m$. \label{fig1}}
	\includegraphics[width=0.5\textwidth]{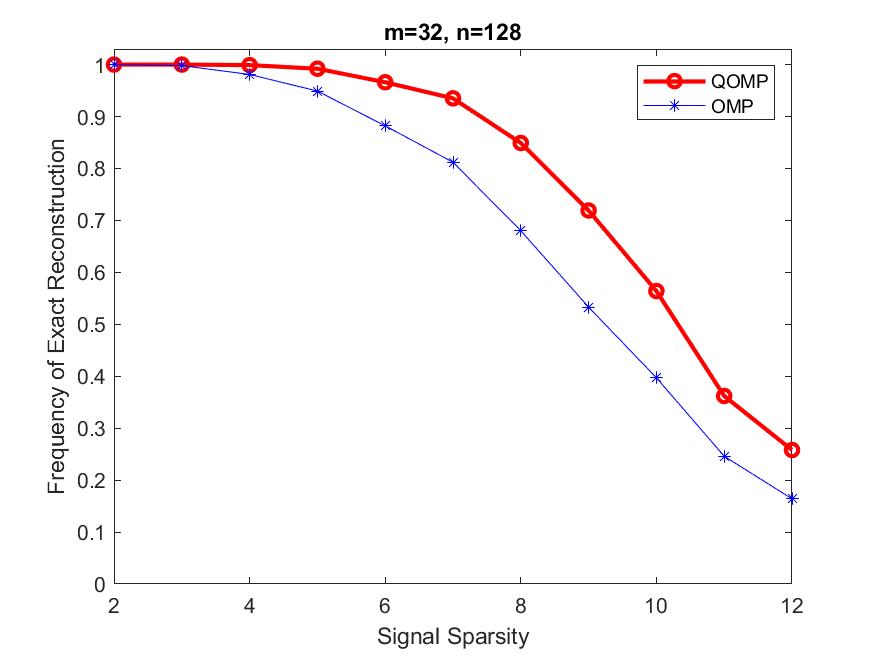}
\end{figure}

However, a total of $s$ iterations in OMP algorithm will only select $s$ different column indices, while in QOMP $2s$ column 
indices are selected. Hence we may want to do $2s$ iterations for OMP as well in order to keep the number of column indices to be 
the same. We also add in the GOMP algorithm into the comparison since QOMP can be considered as a generalization of GOMP with $N=2$
in the sense that QOMP will become GOMP if all the columns are orthogonal to each other.
See Figure~\ref{fig2} for comparison of the performance of QOMP within $s$ iterations, OMP within $2s$ iterations and GOMP (with 
$N=2$) within $2s$ iterations, so that the number of column indices are $2s$ for all three algorithms.
The frequency of exact recovery of each sparsity is computed by each method based on 1000 repetitions of solving a Gaussian random 
matrix of size $m\times n$ with $m=32$ and $n=2m$, $4m$, $6m$ and $8m$, respectively.   From Figure~\ref{fig2} we can see that for 
$n=2m$ the performance of QOMP has no advantage over the standard OMP and GOMP (with N=2). However, QOMP do have a better 
performance for $n=4m$, $6m$, and $8m$, especially for a bigger $n$. 
Empirically speaking, for sensing matrix satisfies $n\leq 2m$, the standard OMP algorithm performs very well if $s<0.2m$. However,  
The performance of OMP and GOMP drops dramatically if $n\geq 4m$. Nevertheless, QOMP has a better performance in this range roughly 
for $2\leq s < 0.4m$. For $s\geq 0.4m$, the frequency of exact recovery of these three algorithms are all very low and hence we d
did not present it in Figure~\ref{fig2}.

\begin{figure}[htpb]
  \centering
  \caption{Frequencies of exact reconstruction of signal within $2s$ iterations (OMP) and $s$ iterations (QOMP and GOMP) for underdetermined linear systems of sizes $m\times n$, where $m=32$, $n=2m$, $4m$, $6m$, $8m$ respectively. \label{fig2}}
   \includegraphics[width=0.4\textwidth]{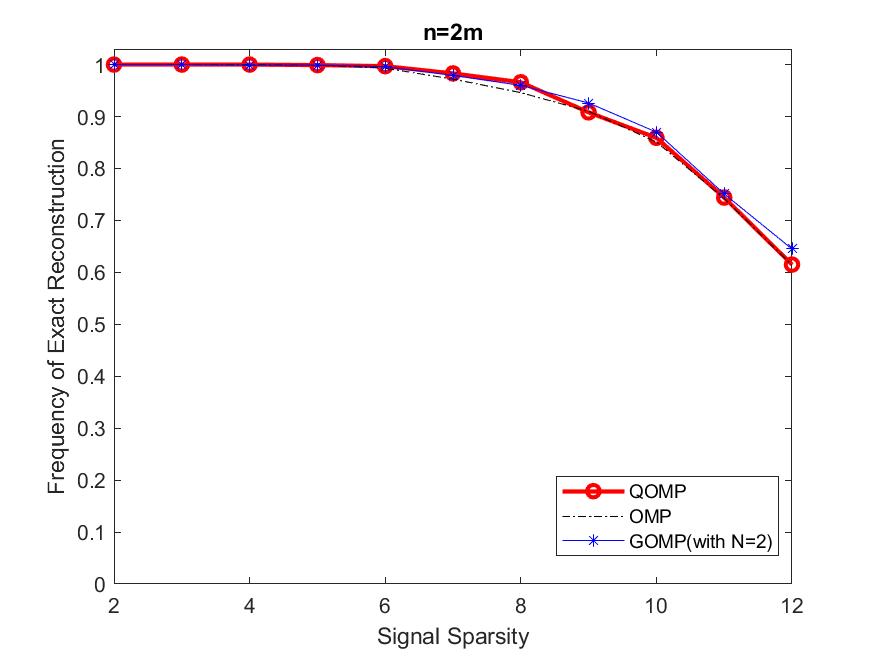}
   \includegraphics[width=0.4\textwidth]{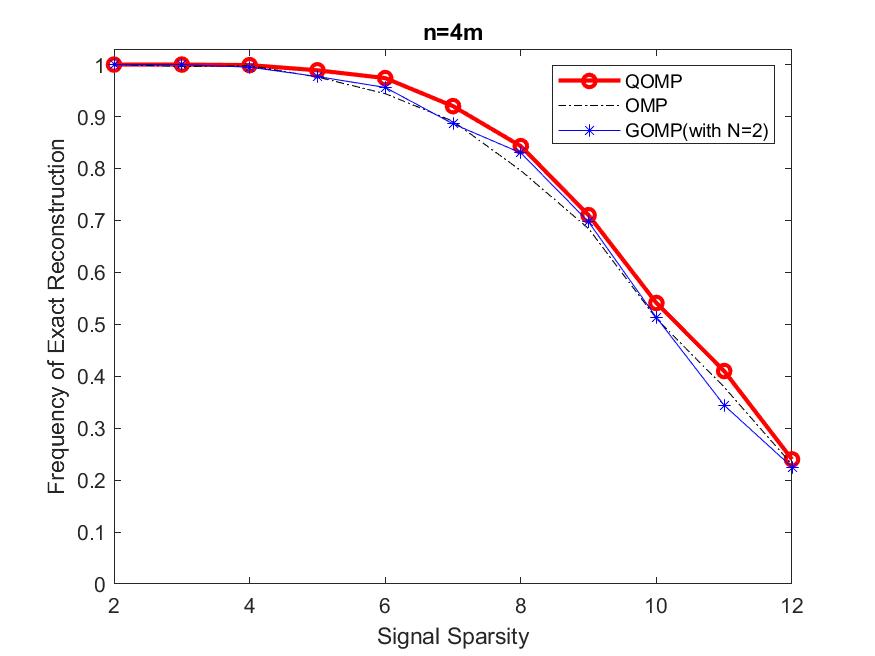}
   \\
   \includegraphics[width=0.4\textwidth]{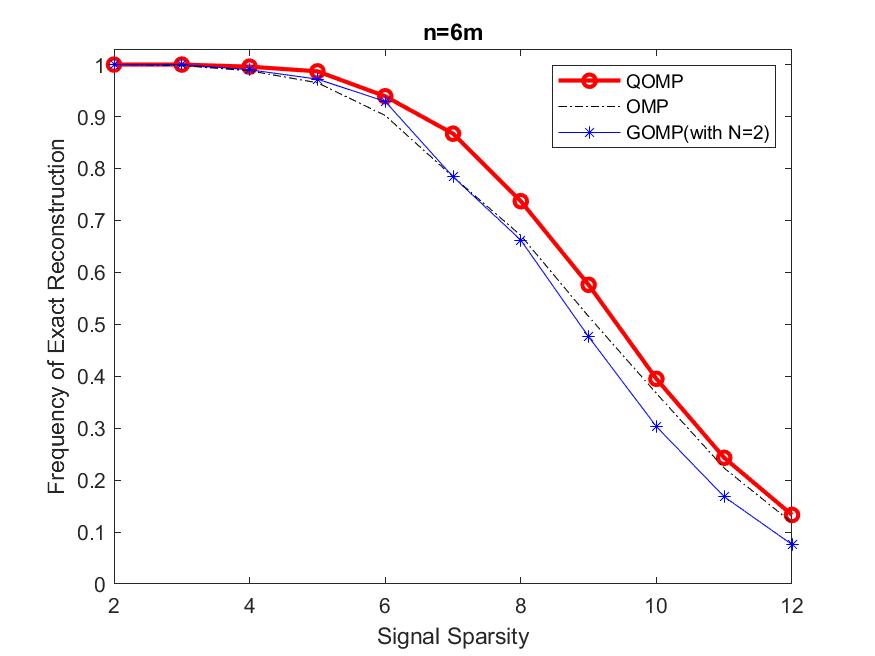}
   \includegraphics[width=0.4\textwidth]{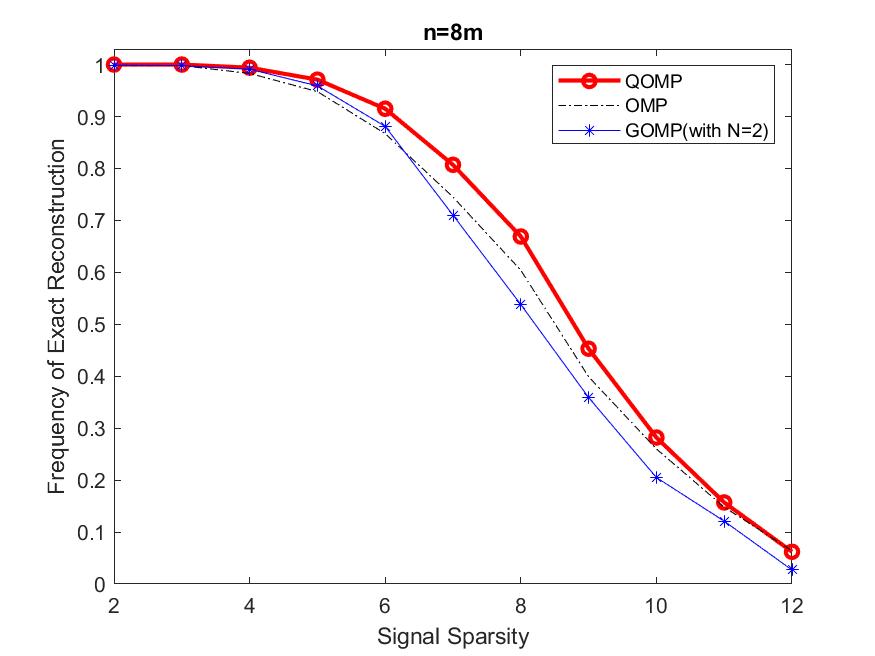}
\end{figure}

\section{Discussions and Future Research}
Future Research could be done to investigate the most optimal relationship between $m$ and $n$ to make QOMP the 
most effective against OMP or GOMP. As we see from our numerical results that a bigger $n$ seems to make QOMP more advantageous 
over OMP and GOMP, however, we are not very sure about the situation for an even bigger $n$ due to the limit of computing power. 

A natural generalization of QOMP algorithm is, in each iteration, to select a $k$-tuple of 
columns of matrix $\mathbf{{\Phi}}\in\mathbb{R}^{m\times n}$ which maximizes the projection of the 
current measurement vector onto the hyperplane which is generated by the most optimal $k$-tuple 
from a total number of $\binom{n}{k}$ $k$-tuples. In this case, if we need $s$ iterations to 
reconstruct the signal $\mathbf{x}$, then a total number of $ks$ columns will be selected after 
$s$ iterations. 
The accuracy of QOMP may or may not go up as $k$ 
increases, however, the computational complexity will increase largely if one increases $k$. For 
example, when $k=3$, the complexity of QOMP becomes $O(mn^3s)$, and even in the case when parallel 
computing or GPU is applied, we can only be able to reduce it to $O(n^3s)$, which is still not optimal compared with the 
standard OMP. 

We showed that if the mutual coherence of sensing matrix satisfies certain conditions, then the total iterations needed to exactly 
recover the $s$-sparse singal $\mathbf{x}$ is $s$. Further research could be done on investigating the other conditions we need to 
impose in order to reduce the number of iterations in QOMP (though the best we can hope is 
$\left\lceil\frac{s}{2} \right\rceil$ iterations). However, in our simulation, almost always 
$\left\lceil\frac{s}{2} \right\rceil$ number of iterations is not enough to guarantee the exact 
reconstruction of $\mathbf{x}$, therefore we expect that the conditions be imposed on 
$\mathbf{\Phi}$ would be quite demanding.

\section{Appendix A: Proof of Lemma \ref{lem1}}
Before proving Lemma \ref{lem1}, let us state some standard results in probability theory, the proofs of Lemma \ref{lem3}, \ref{lem4} and \ref{lem5} can all be referred to \cite{KallenbergProbability}.
\begin{lemma} \label{lem3}
For any $p>0$ and random variables $\xi\geq 0$,
\begin{align*}
\bfE(\xi^p)=p\int_0^{\infty}P(\xi>t)t^{p-1}dt=p\int_0^{\infty}P(\xi\geq t)t^{p-1}dt
\end{align*}
\end{lemma}

\begin{lemma}(Three-series criterion, Kolmogorov, Levy) \label{lem4}
Let $\xi_1,\xi_2,\cdots$ be independent random variables. Then $\sum_n \xi_n$ converges a.s. if and only if it converges in distribution and also if and only if these conditions are fulfilled:
\begin{itemize}
    \item $\sum_n P(|\xi_n|>1)<\infty$;
    \item $\sum_n E[\xi_n;|\xi_n|\leq 1]$ converges;
    \item $\sum_n Var[\xi;|\xi_n|\leq 1]<\infty$.
\end{itemize}
\end{lemma}

\begin{lemma}(Kronecker) \label{lem5}
If $\sum_n n^{-c}a_n$ converges for some $a_1,a_2,\cdots\in\mathbb{R}$ and $c>0$, then $n^{-c}\sum_{k\leq n}a_k\to 0$.
\end{lemma}
Now let us prove Lemma \ref{lem1} by using the above lemmas.
\begin{proof}
\emph{(of Lemma \ref{lem1})} Assume that $\bfE|\xi|^p\leq\infty$ and for $p\geq 1$ that even $\bfE\xi=0$. Define $\xi'_n=\xi_n 1_{\{|\xi_n|\leq n^{1/p}\}}$, and note that by Lemma \ref{lem3},
\begin{align*}
\sum_n P(\xi'_n\neq \xi_n)=\sum_n P(|\xi|^p>n)\leq \int_{0}^{\infty} P(|\xi|^p>t)dt=\bfE|\xi|^p<\infty.
\end{align*}
By the Borel-Cantelli lemma we get $P(\xi'_n\neq \xi, \text{i.o.})=0$, and so $\xi'_n=\xi_n$ for all but finitely many $n\in\mathbb{N}$ a.s.. It is then equivalent to show that $n^{-1/p}\sum_{k\leq n}\xi'_k\to 0$ a.s. By Lemma \ref{lem5} it suffices to prove instead that $\sum_n n^{-1/p}\xi'_n$ converges almost surely. \\
For $p<1$, this is clear if we write 
\begin{align*}
\bfE\Big(\sum_n n^{-1/p}|\xi'_n|\Big)&=\sum_n n^{-1/p}\bfE\big[|\xi|;|\xi|\leq n^{1/p}\big] \\
&\leq \int_0^{\infty}t^{-1/p}\bfE\big[|\xi|;|\xi|\leq t^{1/p}\big]dt \\
&=\bfE\big[|\xi|\cdot\int_{|\xi|^p}^{\infty}t^{-1/p}dt\big] \\
&\leq \bfE|\xi|^p<\infty.
\end{align*} 
If instead $p>1$, it suffices by Lemma \ref{lem4} to prove that $\sum_n n^{-1/p}\bfE(\xi'_n)$ converges and $\sum_n n^{-2/p}Var(\xi'_n)\leq\infty$.
Since $\bfE(\xi'_n)=-\bfE\big[\xi;|\xi|>n^{1/p}\big]$, we have for the former series 
\begin{align*}
\sum_n n^{-1/p}|\bfE(\xi'_n)|&\leq \sum_n n^{-1/p}\bfE\big[|\xi|;|\xi|>n^{1/p}\big] \\
&\leq \int_{0}^{\infty} t^{-1/p\bf}E\big[|\xi|;|\xi|>t^{1/p}\big]dt \\
&=\bfE\big[|\xi|\cdot\int_0^{|\xi|^p}t^{-1/p}dt\big] \\
&\leq \bfE|\xi|^p<\infty.
\end{align*}
As for the latter series, we get 
\begin{align*}
\sum_n n^{-2/p}Var(\xi'_n)&\leq \sum_n n^{-2/p}\bfE(\xi'_n)^2 \\
&=\sum_n n^{-2/p}\bfE\big[\xi^2;|\xi|\leq n^{1/p}\big] \\
&\leq \int_0^{\infty}t^{-2/p}\bfE\big[\xi^2;|\xi|\leq t^{1/p}\big]dt \\
&=\bfE\big[\xi^2\cdot\int_{|\xi^p|}^{\infty}t^{-2/p}dt\big]
\\
&\leq \bfE|\xi|^p<\infty.
\end{align*}
If $p=1$, then $\bfE(\xi'_n)=\bfE\big[\xi;|\xi|\leq n\big]\to 0$ by dominated convergence. Thus, $n^{-1}\sum_{k\leq n}\bfE(\xi'_k)\to 0$, and we may prove instead that $n^{-1}\sum_{k\leq n}\xi''_k\to 0$ a.s., where $\xi''_n=\xi'_n-\bfE(\xi'_n)$. By Lemma \ref{lem5} and Lemma \ref{lem4} it is then enough to show that $\sum_n n^{-2}Var(\xi'_n)<\infty$, which may been seen as before.  \\
Conversely, assume that $n^{-1/p}S_n=n^{-1/p}\sum_{k\leq n}\xi_k$ converges a.s.. Then
\begin{align*}
\frac{\xi_n}{n^{1/p}}=\frac{S_n}{n^{1/p}}-(\frac{n-1}{n})^{1/p}\cdot\frac{S_{n-1}}{(n-1)^{1/p}}\to 0 
\end{align*}
almost surely, and in particular $P(|\xi_n|^p>n, \text{i.o.})=0$. Hence, by Lemma \ref{lem3} and the Borel-Cantelli lemma,
\begin{align*}
\bfE|\xi|^p=\int_0^{\infty}P(|\xi|^p)dt\leq 1+\sum_{n\geq 1}P(|\xi|^p>n)<\infty.
\end{align*}
For $p>1$, the direct assertion yields $n^{-1/p}(S_n-n\bfE(\xi))\to 0$ a.s., and so $n^{1-1/p}\bfE(\xi)$ converges, which implies $\bfE(\xi)=0$.
\end{proof}

\end{document}